\documentclass[12pt]{amsart}
\usepackage{amssymb,amsmath,MnSymbol,xcolor}
\usepackage{comment}
\usepackage{mathtools}
\usepackage{tikz}
\usetikzlibrary{calc}
\usepackage[margin=1.25in]{geometry} 
\usepackage[bookmarks, bookmarksdepth=2, colorlinks=true, linkcolor=blue, citecolor=blue, urlcolor=blue]{hyperref}
\usepackage{enumitem}

\let\ol\overline
\let\ul\underline

\newtheorem{theorem}{Theorem}[section]
\newtheorem{lemma}[theorem]{Lemma}
\newtheorem{proposition}[theorem]{Proposition}
\newtheorem{corollary}[theorem]{Corollary}
\theoremstyle{definition}
\newtheorem{definition}[theorem]{Definition}
\newtheorem{claim}[theorem]{Claim}

\newtheorem{example}[theorem]{Example}

\def\multiset#1#2{\ensuremath{\left(\kern-.3em\left(\genfrac{}{}{0pt}{}{#1}{#2}\right)\kern-.3em\right)}}

\newcommand{\Q}{\mathbb{Q}}

\newcommand{\R}{\mathbb{R}}
\newcommand{\A}{\mathbf{A}}

\newcommand{\cF}{\mathcal{F}}
\newcommand{\cG}{\mathcal{G}}
\newcommand{\cH}{\mathcal{H}}
\newcommand{\cI}{\mathcal{I}}
\newcommand{\cP}{\mathcal{P}}

\renewcommand{\sp}{\textnormal{span}}
\newcommand{\codim}{\textnormal{codim}}

\newcommand{\prk}{\textnormal{prk}}
\newcommand{\grk}{\textnormal{grk}}
\newcommand{\str}{\textnormal{str}}
\newcommand{\brk}{\textnormal{Brk}}

\let\ol\overline
\let\ul\underline

\author{Amichai Lampert}
\address{Department of Mathematics, University of Michigan, Ann Arbor, MI}
\email{\href{mailto:amichai@umich.edu}{amichai@umich.edu}}
\thanks{AL was supported by NSF grant DMS-2402041}

\author{Andrew Snowden}
\email{\href{mailto:asnowden@umich.edu}{asnowden@umich.edu}}
\thanks{AS was supported by NSF grant DMS-2301871}

\author{Tamar Ziegler}
\email{\href{mailto:tamar.ziegler@mail.huji.ac.il}{tamar.ziegler@mail.huji.ac.il}}
\thanks{TZ was supported by ISF grant 697/24 and ERC grant DyAddAlg  (101199203)}

\title{Polynomial bounds for Birch's theorem}
\date{}

\begin{document}

\begin{abstract}
Let $K$ be a number field and $f_1,\ldots,f_s\in K[x_1,\ldots,x_n]$  forms of odd degrees. In 1957, Birch proved that if $n$ is sufficiently large then the forms always have a nontrivial zero in $K^n$. Apart from some small degrees, the number of variables required was so large that it has been described as "not even astronomical". We prove that for any fixed degree, $n$ may be taken as a polynomial in $s$. We deduce this from a stronger result -- the Zariski closure of the set of rational zeros has codimension bounded by a polynomial in $s$. When $K$ is totally imaginary, our results hold for forms of any (possibly even) degrees.
\end{abstract}

\dedicatory{To Steve Schloss, for his 95th birthday}

\maketitle

\section{Introduction}

Let $K$ be a number field and let  $f_1,\ldots,f_s\in K[x_1,\ldots,x_n]$ be forms (i.e. homogeneous polynomials) of \emph{odd} degrees $\le d$ in $n$ variables. We say that the system $\ul{f} = (f_1,\ldots,f_s)$ \emph{has a solution} if there exists $0\neq x\in K^n$ which satisfies $f_i(x) = 0$ for all $1\le i\le s$. In \cite{Birch-odd}, Birch proved that if $n \ge N(d,s)$ then any such $\ul{f}$ has a solution.  The proof relied on an elegant but highly iterative inductive argument which yielded extremely weak bounds.

We slightly abuse notation by now writing $N(d,s)$ for the \emph{minimal} integer for which Birch's theorem holds. Much work has been done on improving the bounds for $N(d,s)$. We mention just a few landmark results for $K=\Q$. The most well-understood case is that of a single cubic form. In a remarkable series of papers, Davenport first proved that $N(3,1) \le 32$ \cite{Dav-32}, then reduced this to 29 \cite{Dav-29} and finally to 16 \cite{Dav-16}. This last record held for over forty years, until it was shattered by Heath-Brown \cite{HB-14} who showed that 14 variables are sufficient. On the other hand, it is known that at least 10 variables are required.

For systems of cubics, Schmidt proved \cite{Sch-cubics} that $N(3,s) \le (10s)^5$. The dependence on $s$ was subsequently improved to quartic by Dietmann \cite{Diet-cubics}. For systems of quintic forms, the best bound available is due to Wooley \cite{Wool-quintic} and is approximately $ c^{s^6}$ for a very large constant $c$. For systems of forms of larger odd degrees, the only explicit bounds available are of recursive type and are also due to Wooley \cite{Wool-explicit}. They are so large that Wooley describes them as "not even astronomical".

Our main result is that $N(d,s)$ is bounded by a polynomial in $s$, for \emph{all} $d$. In fact, we prove substantially more. Writing $Z = (f_1=\ldots=f_s = 0) \subset \A^n$, we have:

\begin{theorem}\label{thm:main-num}
    For any $d$ there exist constants $A = A_{\ref{thm:main-num}}(d),B = B_{\ref{thm:main-num}}(d)$ such that $\codim_{\A^n} \overline{Z(K)} < As^B$, where the closure is taken in the Zariski topology. In particular, $N(d,s) \le As^B$.   
\end{theorem}

In our discussion so far, the assumption that the $f_i$ have odd degrees was a natural way to avoid positive definite forms such as $x_1^2+\ldots+x_n^2$ over $\Q$. When $K$ is totally imaginary (for example $K = \Q[i]$) this is no longer an issue. In this setting, we will see that theorem \ref{thm:main-num} applies for even degrees as well. Actually, it will hold for a much larger class of fields. 

\begin{definition}\label{def:Brauer}
    $K$ is called a \emph{Brauer field} if for any $d\ge 1$ there exists an integer $\phi  = \phi_d$ such that $a_1x_1^d+\ldots+a_\phi x_\phi^d $ has a solution, for any $a_1,\ldots,a_\phi\in K$.
\end{definition}

Birch proved that $p$-adic fields are Brauer \cite{Birch-p-adic} and Peck proved this for totally imaginary number fields \cite{Peck}. See \cite{BDS} for many more examples of Brauer fields. Brauer famously proved \cite{Brauer} that for every Brauer field $K$ there exists $n_0(d,s)$ such that any collection of forms $f_1,\ldots,f_s\in K[x_1,\ldots,x_n]$ of degrees $\le d$ with $n\ge n_0$ has a solution. As in Birch's theorem, the proof was based on a highly iterative induction and yielded very weak bounds for $n_0(d,s)$. This remained the state of affairs for many years until Leep and Schmidt \cite{LS} introduced an efficient modification of Brauer's inductive argument, obtaining a value of $n_0(d,s)$ that was polynomial in the parameters $s,\phi_2,\ldots,\phi_d$. Wooley refined their inductive scheme to yield a further improvement in \cite{Wooley}. For Brauer fields we obtain an improvement analogous to theorem \ref{thm:main-num}.

Our main result is as follows.

\begin{theorem}\label{thm:main-Brauer}
For any $d$ there exist constants $A = A_{\ref{thm:main-Brauer}}(d), B = B_{\ref{thm:main-Brauer}}(d)$ such that the following holds. If $K$ is a Brauer field of characteristic zero and $f_1,\ldots,f_s\in K[x_1,\ldots,x_n]$ are forms of degree $\le d$ then
\[
\codim_{\A^n} \overline{Z(K)} \le A(s+\phi_2+\ldots+\phi_d)^B.
\]
In particular, by a bound from \cite{Wooley}, if $K$ is a totally imaginary number field we get
\[
\codim_{\A^n} \overline{Z(K)} \le A (s+e^{2d+1})^B \le A's^B.
\]
\end{theorem}

Theorems \ref{thm:main-num} and \ref{thm:main-Brauer} will follow from relative regularization combined with the next two results (we define all of the relevant notions in the next section).

\begin{theorem}\label{thm:dense}
    For any $d$ there exist constants $C = C_{\ref{thm:dense}}(d), D = D_{\ref{thm:dense}}(d)$ such that the following holds. If $K$ is a Brauer field of characteristic zero and $\cF$ is absolutely $(C,D, \phi_2+\ldots+\phi_d)$-strong then the $K$-points of $Z(\cF)$ are Zariski dense. The same conclusion holds when $K$ is a number field, $\cF$ is composed of forms of odd degrees and is absolutely $(C,D,1)$-strong.
\end{theorem}

It is instructive to compare this with results of a somewhat similar spirit which are obtained via the Hardy-Littlewood circle method. For systems of sufficiently nonsingular forms, asymptotic formulas were obtained for the number of integer solutions in \cite{Birch,Skinner-HLS,BHB,FM}. Schmidt \cite{Schmidt} used this to deduce lower bounds for the asymptotic point count of integer solutions for general systems of equations, but with very weak dependence on $d,s$.  Our hypothesis on $\cF$ is substantially more flexible than the one required in the cited results, and we will see in proposition \ref{prop:rel-reg} how this flexibility allows us to obtain such dramatic improvements for the  bounds of theorems \ref{thm:main-num} and \ref{thm:main-Brauer}.

\begin{theorem}\label{thm:reg}
    Given any $C,D,d$ there exist constants $C' = C'_{\ref{thm:reg}}(d,A,B),D' = D'_{\ref{thm:reg}}(d,A,B)$ such that the following holds. If $K$ is a field of characteristic zero and $\cF$ is a tower of forms in $K[x_1,\ldots,x_n]$ of degree $(1,\ldots,d)$ which is $(C',D',r)$-strong then it is absolutely $(C,D,r)$-strong. 
\end{theorem}

In the next section we quickly deduce theorems \ref{thm:main-num} and \ref{thm:main-Brauer} from these two results. After that we collect some geometric properties of regular towers and prove theorem \ref{thm:dense}. Then we turn to theorem \ref{thm:reg}, the technical crux of the paper.

\section{Deduction of theorems \ref{thm:main-num} and \ref{thm:main-Brauer}}

In this section we assume that theorems \ref{thm:dense} and \ref{thm:reg} hold. We begin with the definition of strength, introduced by Schmidt \cite{Schmidt}.

\begin{definition}
    The \emph{strength} of a form $f\in K[x_1,\ldots,x_n]$ is
    \[
    \str(f) = \inf \{r: f = g_1h_1+\ldots+g_rh_r\},
    \]
    where $g_i,h_i\in K[x_1,\ldots,x_n]$ are forms of lower degree. Note that linear forms have infinite strength. More generally, the strength of a collection of forms $f_1,\ldots,f_s$ of the same degree is 
    $$\str(f_1,\ldots,f_s) = \min \{ \str(a_1f_1+\ldots+a_sf_s) : a_1,\ldots,a_s\in K \textnormal{ not all zero}\}.$$
    The \emph{absolute strength} of $f$, denoted $\ol{\str}(f)$, is its strength in $\ol{K}[x_1,\ldots,x_n]$, where $\ol{K}$ is an algebraic closure of $K$. 
\end{definition}

The above definition generalizes to any standard graded algebra $R = K[x_1,\ldots,x_n]/I$, where $I$ is a homogeneous ideal. This generalization was first explored in \cite{LZ-rel}, where it was called relative rank.

\begin{definition}
    Given a form $f\in K[x_1,\ldots,x_n]$ of degree $d$, its \emph{relative strength} is 
    \[
     \str_I(f) = \inf \{r: f = g_1h_1+\ldots+g_rh_r \mod I\},
    \]
    where $g_i,h_i\in K[x_1,\ldots,x_n]$ are forms of lower degree. When $f_1,\ldots,f_s \in K[x_1,\ldots,x_n]$ are forms of the same degree their collective relative strength is
    $$\str_I(f_1,\ldots,f_s) = \min \{ \str(a_1f_1+\ldots+a_sf_s) : a_1,\ldots,a_s\in K \textnormal{ not all zero}\}.$$
\end{definition}

In \cite{Schmidt,Mil}, it is shown that good properties hold for collections of forms which have strength larger than some polynomial expression in the number of forms involved and an error parameter. For relative strength we will work with a similar hypothesis, described in the following definition.  

\begin{definition}
\begin{enumerate}
    \item A \emph{tower} $\cF$ of degree $\ul{d} = (d_1,\ldots,d_h)$ in $R$ is a collection of forms $\cF = (\cF_i)_{i\in [h]}$ where each  $\cF_i = (f_{i,j})_{j\in [s_i]}$ is a tuple of homogeneous elements of $R$ of degree $d_i.$ For $i\in [h]$ we denote by $\cF_{<i}$ and $\cF_{\le i}$ the tower obtained by keeping only those forms with $i'<i$ or $i'\le i,$ respectively.
    \item We say that the tower $\cF$ is $(A,B,r)$-\emph{strong} if for each $i\in [h]$ we have  
    \[
    \str_{\cF_{<i}} (\cF_i) > A(s_i+\ldots+s_h+r)^B,
    \]
    where $\str_{\cF_{<i}}$ denotes the strength in the quotient $R/(\cF_{<i})$.  It is \emph{absolutely} $(A,B,r)$-strong if $\cF\otimes \overline{K}$ is $(A,B,r)$-strong.
\end{enumerate}

\end{definition}

A key step in various applications of strength (e.g. \cite{Schmidt,BDS}) is a regularization process where an arbitrary collection of forms is replaced by a high strength collection which has some desirable property. This regularization process is extremely inefficient, leading to bounds of highly recursive type in any result where it is used. In \cite{LZ-rel} it was shown that relative strength retains many of the good properties of strength, while also naturally lending itself to an \emph{efficient} regularization process, allowing one to replace an arbitrary initial collection of forms $f_1,\ldots,f_s$ with a strong tower of size \emph{polynomial} in $s.$

\begin{proposition}\label{prop:rel-reg}
Given any $C,D,d$ there exist constants $E = E(C,D,d), F = F(C,D,d)$ such that the following holds. For any homogeneous $f_1,\ldots,f_s$ of positive degrees $\le d$ there exists a $(C,D,r)$-strong tower $\cF$ of degree $(1,\ldots,d)$ and size $s_1+\ldots+s_h \le E(s+r)^F$ such that the ideals satisfy the containment $(f_1,\ldots,f_s) \subset (\cF).$ Moreover, if the $f_i$'s all have odd degrees then we can choose $\cF$ to be composed only of homogeneous elements of odd degree.    
\end{proposition}

\begin{proof}
    We begin with an initial tower $\cF^{(0)}$ obtained by placing $f_i$ in the $\deg(f_i)$ layer.  To each layer $i\in [d]$ we assign a number $n_i$ by a formula we'll soon describe. Now we regularize: If for each $i\in[d]$ we have $\str_{\cF_{<i}} (\cF_i) > Cn_i^D$ then we're done.
    
    Otherwise, there exists some $i$ and $a_1,\ldots,a_{s_i}\in K$ not all zero with 
    \begin{equation}\label{eq:rel-reg}
        \sum_{j=1}^{s_i} a_j f_{i,j} = \sum_{k=1}^t p_kq_k \mod (\cF_{<i}),
    \end{equation}
    where $t\le Cn_i^D.$ Assume without loss of generality that $a_1\neq 0$ and replace $\cF^{(k)}$ by the tower $\cF^{(k+1)}$ obtained by deleting $f_{i,1}$ and by adding each $p_k$ to the layer of degree $\deg(p_k)<i.$ This process must eventually terminate, yielding $\cF^{(k_0)} = \cF^{(\infty)} $ which is $(C,D)$-strong and $(f_1,\ldots,f_s) \subset (\cF).$ Now we give the formula for $n_i$ in terms of $m_i,$ the original number of forms in each degree,
    \[
    n_d = m_d\ ,\ n_i = m_i+Cn_{i+1}^{D+1}.
    \]
    Note that $|n_i| \ge \bigcup_{k=0}^\infty \cF_{\ge i}^{(k)}.$ This can be seen by downward induction on $i$, where the base case $i=d$ is clear. Assuming it holds for $i+1,$ we get that each of the $n_{i+1}$ forms appearing in degree $i+1$ and above can contribute at most $Cn_{i+1}^D$ forms of degree $i,$ proving the claim for degree $i.$  It's easy to see that all of the $n_i$ are polynomial in $s$ and so $s_1+\ldots+s_h\le n_1 \le Es^F$ as desired. For the statement about odd degrees, note that in equation \eqref{eq:rel-reg} either $p_k$ or $q_k$ must have odd degree, so assume without loss of generality that $p_k$ has odd degree.  
\end{proof}

\begin{proof}[Proof of theorems \ref{thm:main-num} and \ref{thm:main-Brauer}]
    We first prove theorem \ref{thm:main-Brauer}. Let $C,D$ be the constants of theorem \ref{thm:dense} and let $C',D'$ be the corresponding constants of theorem \ref{thm:reg}. Apply proposition \ref{prop:rel-reg} to replace $\ul{f}$ by a $(C',D',\phi_2+\ldots+\phi_d)$-strong tower $\cF$ of size at most $E(s+\phi_2+\ldots+\phi_d)^F.$ By theorem \ref{thm:reg}, $\cF$ is absolutely $(C,D,\phi_2+\ldots+\phi_d)$-strong. Let $Z = (f_1 = \ldots =f_s = 0)$ and $Z' = (\cF = 0).$ By theorem \ref{thm:dense}, $\overline{Z'(K)} = Z'$, which implies 
    \[
    \codim_{\A^n} \overline{Z(K)} \le \codim_{\A^n} \overline{Z'(K)} =  \codim_{\A^n} \overline{Z'} \le E(s+\phi_2+\ldots+\phi_d)^F.
    \]
    This proves theorem \ref{thm:main-Brauer} (with $A = E, B = F$).
    The proof of theorem \ref{thm:main-num} is identical, except that proposition \ref{prop:rel-reg} is applied for odd degrees. 
\end{proof}

\section{Preliminaries}

In this section we collect various results required for the proof of theorem \ref{thm:dense}.

\subsection{Relative strength and singularities}

In \cite{Birch}, Birch introduced the following definition.

\begin{definition}
    The \emph{Birch rank} of a form $f$ is $\brk(f) = \codim_{\A^n} (x:\nabla f(x) = 0).$ More generally, the Birch rank of a collection of forms $f_1,\ldots,f_s$ is 
    \[
    \brk(f_1,\ldots,f_s) = \codim_{\A^n} (x:\nabla f_1(x),\ldots,\nabla f_s(x) \textnormal{ are linearly dependent}).
    \]
\end{definition}

Birch rank and absolute strength are closely related by the following result.  
\begin{lemma}[Kazhdan-Lampert-Polishchuk \cite{KLP}]\label{lem:str-brk}
    \[
    \frac{\brk(f_1,\ldots,f_s)}{2} \le \ol{\str}(f_1,\ldots,f_s) \le (d-1) (\brk(f_1,\ldots,f_s) +s-1)
    \]
\end{lemma}

A corresponding result holds for relative strength. Before stating it, we note some simple consequences of the definition of strong towers.

\begin{claim}\label{str-shuffle} \mbox{}
    \begin{enumerate}
        \item Given $A,B$ there exist $C = C(A,B),D = D(A,B)$ such that if $\cF$ is $(C,D,r)$-strong then it is $(A(r+s_h)^B,B,A(r+s_h)^B)$-strong.
        \item If $\cF$ is $(A,B,r+|\cG|)$-strong and $\cG \mod (\cF)$ is $(A,B,r)$-strong then the tower $(\cF,\cG)$ is $(A,B,r)$-strong (the notation indicates that the layers of $\cG$ are placed above those of $\cF$).
    \end{enumerate}
\end{claim}

\begin{proof}
    The second statement follows immediately from the definitions. For the first statement, since 
    \begin{align*}
        A(r+s_h)^B(s_i+\ldots+s_h+A(r+s_h)^B)^B &\le A(r+s_h)^B(2A(s_i+\ldots+s_h+r)^B)^B \\
        &\le 2^B A^{B+1} (s_i+\ldots+s_h+r)^{B^2+B},
    \end{align*}
    we can take $C = 2^B A^{B+1}$ and $D = B^2+B$.
\end{proof}

Relative strength and singularities are closely related as follows.  

\begin{lemma}[Lampert \cite{L-small}]\label{lem:str-reg}
    If $\cF$ is absolutely $(A,B,r)$-strong then 
    \[
    \codim_{Z(\cF_{<h})} Z(\cF_{<h})\cap S(\cF) > r,
    \]
    where $S(\cF) = (x: \nabla f_{i,j} \textnormal{ are linearly dependent})$. 
\end{lemma}

As a consequence, we have: 

\begin{lemma}\label{lem:str-prime}
    There exist $C,D$ such that if $\cF$ is absolutely $(C,D,1)$-strong then $Z(\cF)$ is a complete intersection and $(\cF)$ is prime.
\end{lemma}

The proof of the above lemma relies on a standard criterion for primeness of ideals (see \cite[Theorem 18.15 and Proposition 18.13]{Eisenbud}).

\begin{lemma}
    Suppose that $f_1,\ldots,f_s$ are polynomials which cut out a complete intersection $Z = (f_1=\ldots = f_s = 0)\subset \A^n$ and let  $R = K[x_1,\ldots,x_n]/(f_1,\ldots,f_s)$ be its coordinate ring. If $\codim_Z Z\cap S(\ul{f}) >1$ then $R$ is a direct product of domains.
\end{lemma}

\begin{corollary}\label{cor:reg-prime}
    Under the hypotheses of the previous lemma, if $f_1,\ldots,f_s$ are homogeneous, then $(f_1,\ldots,f_s)$ is prime.
\end{corollary}

\begin{proof}
    By the lemma, $R$ is a direct product of domains. Since it is also a graded ring over a field, it must be a domain. 
\end{proof}

\begin{proof}[Proof of lemma \ref{lem:str-prime}]
    
    The proof is by induction on $h$, where the base case $h=0$ is trivial. Suppose the statement holds for towers of height $< h$ and let $\cF$ be a tower of height $h$. By claim \ref{str-shuffle}, we may assume that $\cF$ is $(A,B,s_h+1)$-regular, after perhaps increasing $C,D$.  
    
    We first prove that $Z(\cF)$ is a complete intersection.
    Suppose, to get a contradiction, that $X\subseteq Z(\cF)$ is an irreducible component of codimension $<s.$ Then  $X\subseteq S$. Writing $Z = Z(\cF)$ and $Z_{<h} = Z(\cF_{<h})$,  We obtain the contradiction
    \[
    s_h+1 < \codim_{Z_{<h}} S\cap Z_{<h} \le \codim_{Z_{<h}} X < s_h.
    \]

    Now we move on to primeness of $(\cF)$. By lemma \ref{lem:str-reg} we have
    \[
    \codim_Z Z\cap S = \codim_{Z_{<h}} Z\cap S -s_h \ge \codim_{Z_{<h}} Z_{<h} \cap S -s_h > 1.  
    \]
    By corollary \ref{cor:reg-prime}, $(\cF)$ is indeed prime. 
\end{proof}

\subsection{Taylor expansion and restriction to subspaces} We now recall the multi-homogeneous Taylor expansion for forms, which will be used to investigate their restriction to subspaces. 

\begin{definition}
    The \emph{Taylor expansion} of a  degree $d$ form $f$ is given by 
    $$f(x_1+\ldots+x_m) = \sum_{|e|=d} f^e(x_1,\ldots,x_m),$$ where each $f^e$ is multi-homogeneous of multi-degree $e$.
\end{definition}

\begin{lemma}\label{lem:Taylor-identities} 
    The terms of the Taylor expansion satisfy
    \[
    f^e(x,\ldots,x) = \binom{d}{e_1,\ldots,e_m} f(x).
    \]
        
        
\end{lemma}

\begin{proof}
    Compute the coefficient of $t_1^{e_1}\ldots t_m^{e_m}$ in the expression $f(t_1x+\ldots+t_m x)$ in two different ways. On the one hand, $f(t_1x+\ldots+t_m x) = (t_1+\ldots+t_m)^d f(x)$ so this coefficient is $\binom{d}{e_1,\ldots,e_m} f(x)$. On the other hand, 
    \[
    f(t_1x+\ldots+t_m x) = \sum_{|\alpha| = d} f^\alpha (t_1 x,\ldots,t_m x) = \sum_{|\alpha| = d} t_1^{\alpha_1} \ldots t_m^{\alpha_m} f^\alpha (x,\ldots,x),
    \]
    so this coefficient equals $f^e (x,\ldots,x)$.
\end{proof}

\begin{definition}
    Let $V = K^n$ and $\cF$  a tower in $K[V]$. For any positive integer  $m$ we have a tower $T_m(\cF)$ in $K[V^m]$ where the elements of $(T_m(\cF))_i$ are given by the terms of the Taylor expansion of the elements of $\cF_i$. Note that the forms in $T_m(\cF)_i$ have the same degree as those in $\cF_i$.  
\end{definition}

\begin{lemma}\label{lem:Taylor-strong}
    If $\cF$ is $(Am^{dB},B,r)$-strong then $T_m(\cF)$ is $(A,B,r)$-strong.
\end{lemma}

\begin{proof}
  The number of forms in $(T_m(\cF))_i$ is $\le m^d s_i$, so it's enough to show that the relative strength is at least that of $\cF$. If $\cF_i = (f_1,\ldots,f_s)$ let $g = \sum_{j\in [s], |e| = d_i} a_{j,e}f^e_j $ be a nontrivial linear combination with $\str_{T_m(\cF)_{<i}} (g) = r.$  For $\lambda = (\lambda_1,\ldots,\lambda_m) \neq 0,$ consider the diagonal embedding given by  $x\mapsto (\lambda_1 x,\ldots,\lambda_m x).$ By lemma \ref{lem:Taylor-identities} we have $f^e(\lambda_1 x,\ldots,\lambda_m x) = \binom{d}{e_1,\ldots,e_m} \lambda^e f(x)$ and thus plugging this in to the relative rank of $g$ gives us 
  \[
  \str_{\cF_{<i}} \left( \sum_j f_j(x) \sum_e a_{j,e} \binom{d}{e_1,\ldots,e_m} \lambda^e  \right) \le r. 
  \]
    Our assumption that $g$ is a nontrivial linear combination implies that for some $j$ not all $a_{j,e}$ are zero. Therefore, there exists a choice of $\lambda$ such that $\sum_e a_{j,e} \binom{d}{e_1,\ldots,e_m} \lambda^e$ is not zero, which implies that $\str_{\cF_{<i}} (\cF_i) \le r$ as desired. 
\end{proof}

For convenience, we write $Z_m(\cF)$ for $Z(T_m(\cF))$.
\begin{lemma}\label{lem:res-var}
    Let $m \ge s_h+2,$ and suppose $\cF$ is absolutely $(Cm^{dD},D,1)$-strong. Then the variety 
    \[
    X = \left( (a,\ul{v})\in \A^m\times V^m: \ul{v}\in Z_m(\cF_{<h}),\ a_1v_1+\ldots+a_mv_m \in Z(\cF_h) \right)
    \]
    is an irreducible complete intersection of codimension $|T_m(\cF_{<h})|+s_h$.
\end{lemma}

\begin{proof}
    By the previous lemma, $T_m(\cF)$ is absolutely $(C,D,1)$-strong. For fixed $0 \neq a\in \A^m$, note that the $s_h$ forms $\cF_h (a\cdot \ul{v})$ are a linearly independent subset of $T_m(\cF_h)$. Thus, the equations for the fiber $X(a)$ are absolutely $(C,D,1)$-strong and hence by lemma \ref{lem:str-prime} this fiber is an irreducible complete intersection. The remaining fiber over zero has codimension $|T_m(\cF_{<h})| + m \ge |T_m(\cF_{<h})|+s_h$, so that $X$ is indeed a complete intersection of the claimed codimension.

    Now we prove that $X$ is irreducible. Let  By corollary \ref{cor:reg-prime}, this will follow once we show that $\codim X\cap S > |T_m(\cF_{<h})|+s_h+1$, where $S$ is the set where the differentials of the defining equations for $X$ are linearly depdendent.  The fiber over zero is $Z_m(\cF_{<h}))$ and has codimension $ m+|T_m(\cF_{<h})| > |T_m(\cF_{<h})|+s_h+1$, since $Z_m(\cF_{<h})$ is a complete intersection by lemma \ref{lem:str-prime}. For any $a\neq 0$, the fiber $X(a)\cap S(a)$ is contained in $Z_m(\cF_{<h})\cap S(T_m(\cF))$. By claim \ref{str-shuffle}, for $C(\ul{d}),D(\ul{d})$ sufficiently large we have that $T_m(\cF)$ is absolutely $(A,B,m+s_h+1)$-strong and hence $\codim Z_m(\cF_{<h})\cap S(T_m(\cF)) > |T_m(\cF_{<h})| +m+s_h+1$. Taking the union over the fibers $X(a)\cap S(a)$, we get $\codim X\cap S > |T_m(\cF_{<h})|+s_h+1$,  completing the proof.
\end{proof}

We also need the fact that a generic tuple of forms has large absolute strength.
\begin{lemma}\label{lem:gen-str}
     For an open dense set of tuples of forms $f_1,\ldots,f_s$ of degree $d$ in $m \ge d$ variables we have  $\overline{\str}(f_1,\ldots,f_s) > \frac{m-s}{d2^d}.$ 
\end{lemma}

\begin{proof}
    This is an effective version of \cite[Proposition 2.7]{LS-Birch}. Let $r = \frac{m-s}{d2^d}$. First, note that the set of $f$ with absolute strength $\le r$ is parametrized by $r$ pairs of forms $g_i,h_i$ of degrees $a_i,b_i < d$ such that  $a_i + b_i =d$.  For a fixed sequence $a_1,\ldots,a_r < d$, the dimension of this parameter space is 
    \[
    \sum_{i=1}^r \left[\binom{m+a_i-1}{a_i} + \binom{m+b_i-1}{b_i}\right]  \le r\left[\binom{m+d-2}{d-1} + m \right].
    \]
    For fixed $a_1,\ldots,a_s\in P^{s-1}$, the codimension of $f_1,\ldots,f_s$ such that $a_1f_1+\ldots+a_sf_s $ has absolute strength $\le r$ is thus $\ge \binom{m+d-1}{d} - r(\binom{m+d-2}{d-1} + m). $ By a union bound, the codimension of $f_1,\ldots,f_s$ with $\overline{\str}(f_1,\ldots,f_s) \le r$ is at least  
    \begin{align*}
        \binom{m+d-1}{d} - r\left(\binom{m+d-2}{d-1} + m\right) - (s-1) &\ge \frac{m^d}{d!} -2 r \frac{(m+d-2)^{d-1}}{(d-1)!} - (s-1) \\
        &\ge \frac{m^d}{d!} - r\frac{2^d m^{d-1}}{(d-1)!} - (s-1) \\
        &= \frac{sm^{d-1}}{d!} - (s-1) \ge 1,
    \end{align*}
    as desired. 
\end{proof}

\section{Proof of theorem \ref{thm:dense}}

The proof of theorem \ref{thm:dense} is by induction on $h,$ the number of layers in $\cF.$

\subsection{Base case $h=1$}
For Brauer fields, we rely on the following result from \cite{Lam-density}.

\begin{proposition}
    There exist constants $A(d),B(d)$ such that if $K$ is a Brauer field, $f_1,\ldots,f_s$ all have degree $d$ and  $\brk(f_1,\ldots,f_s) > A(s+\phi_2+\ldots+\phi_d)^B $ 
    then the $K$-points of $Z(\ul{f})$ are Zariski dense.
\end{proposition}

\begin{corollary}\label{Brauer-base}
    Theorem \ref{thm:dense} holds over Brauer fields for towers of height one.
\end{corollary}

\begin{proof}[Proof of corollary \ref{Brauer-base}]
    This is an immediate consequence of the above proposition and lemma \ref{lem:str-brk}. 
\end{proof}

For number fields, we begin with the following well-known fact.
\begin{lemma}\label{lem:real-pts}
    If $f_1,\ldots,f_s \in \R[x_1,\ldots,x_n]$ are forms of odd degrees and $Z = Z(f_1,\ldots,f_s)$ is an irreducible complete intersection, then the real points are Zariski dense in $Z$. 
\end{lemma}

\begin{proof}
    The Zariski closure $Y = \ol{Z(\R)}$ has dimension at least $n-s = \dim Z$.
\end{proof}

Our main tool is a result of Skinner on weak approximation. First we recall the definition.

\begin{definition}
    Let $K$ be a number field and $f_1,\ldots,f_s\in K[x_1,\ldots,x_n]$ forms. We say that \emph{weak approximation} holds for $Z = Z(f_1,\ldots,f_s)$ if $Z(K) \neq \{0\}$ and the diagonal embedding $Z(K)\to \prod_v Z(K_v)$ has dense image, where the product is over the completion of $K$ at all places $v$.
\end{definition}

Now we can state Skinner's result \cite[Corollary 1]{Skinner-HLS}.

\begin{proposition}[Skinner]\label{prop:skinner}
    Suppose that $f_1,\ldots,f_s$ all have degree $d$ and that 
    \[
    \brk(f_1,\ldots,f_s) > s(s+1)(d-1)2^{d-1}.
    \]
    Suppose furthermore that  $Z = Z(f_1,\ldots,f_s) $  is a complete intersection and that it has a non-singular $K_v$-point for every place $v$. Then weak approximation holds for $Z$. 
\end{proposition}

\begin{corollary}\label{base-num}
    Theorem \ref{thm:dense} holds over number fields for towers of height one.
\end{corollary}

\begin{proof}[Proof of corollary \ref{base-num}]

    Choosing $A(d),B(d)$ sufficiently large in the hypothesis of theorem \ref{thm:dense}, lemma \ref{lem:str-brk}  guarantees that $\brk(f_1,\ldots,f_s) > s(s+1)(d-1)2^{d-1}$ and also that  $Z$ is an irreducible complete intersection by lemma \ref{lem:str-prime}.  Now, by lemma \ref{lem:real-pts},  it is enough to prove that weak approximation holds for $Z$. By proposition \ref{prop:skinner}, this will follow once we show that the $K_v$-points of $Z$ are Zariski dense for every place $v$. For infinite places, this is immediate from lemma \ref{lem:real-pts}.  
    
    For finite places, Zariski-density of the $K_v$-points of $Z$ follows from corollary \ref{Brauer-base} for Brauer fields, together with the fact that $\phi_d \le \Gamma (d)$ for any finite extension of $\Q_p$ (see \cite{Birch-p-adic}).
    
\end{proof}

\subsection{The inductive step}

Now that we have established theorem \ref{thm:dense} for towers of height one, we move on to the inductive step. We will use the following standard fact.

\begin{lemma}\label{reg-to-surj}
    If $f_1,\ldots,f_s\in K[x_1,\ldots,x_n]$ are forms which comprise a regular sequence then the map $\ul{f}:\overline{K}^n\to \overline{K}^s$ is surjective.
\end{lemma}

\begin{proof}
    Let $a_1,\ldots,a_s$ be scalars and consider the forms $g_i = f_i-a_iz^{d_i},$ where $z$ is a new variable and $d_i$ is the degree of $f_i$. Then $z,g_1,\ldots,g_s$ is a regular sequence of forms, hence a regular sequence in any order, so $g_1,\ldots,g_s,z$ is also a regular sequence. In particular, $z$ doesn't vanish identically on $Z(g_1,\ldots,g_s).$  If $(x,z)$ is a point with $g_i(x,z) = 0$ and $z\neq 0$ then we have $f_i(x/z) = a_i$ as desired. 
\end{proof}

The inductive step relies on the following proposition.

\begin{proposition}
    Suppose $m\ge \max(d,s_h+2)$, $\cF$ is absolutely $(Am^{dB},B,1)$-strong and $U\subseteq Z(\cF)$ is a nonempty open subset. Then there is a dense open $W\subseteq Z_m(\cF_{<h})$ such that, for all $(v_1,\ldots,v_m)\in W$, the subspace $V_0 = \sp(v_1,\ldots,v_m)$ satisfies
    \begin{enumerate}
        \item $\overline{\str}(\cF_h\restriction_{V_0}) > \frac{m-s_h}{d2^d}$ and
        \item $Z(\cF_h) \cap U\cap V_0 \neq \emptyset.$
    \end{enumerate}
\end{proposition}

\begin{proof}
    By lemmas \ref{lem:Taylor-strong} and \ref{lem:str-prime}, $Z_m(\cF_{<h})$ is irreducible. Thus, it suffices to verify that each of the conditions holds separately on a nonempty open set. By possibly shrinking $U$, we may assume that $U = (x\in Z(\cF):g(x)\neq 0)$ for some polynomial $g$. The variety 
    \[
     X = \left( (a,\ul{v})\in \A^m\times V^m: \ul{v}\in Z_m(\cF_{<h}),\ a_1v_1+\ldots+a_mv_m \in Z(\cF_h) \right)
    \]
    is irreducible by lemma \ref{lem:res-var}. Note also that $X$ surjects onto $Z_m(\cF_{<h})$ by projection on $\ul{v}$ and that the polynomial $g'(\ul{a},\ul{v}) := g(\ul{a}\cdot \ul{v})$ does not vanish identically on $X$, hence it cuts down the dimension. It follows that for a dense open set of $\ul{v} \in Z_m(\cF_{<h}),$ $g'(\cdot,\ul{v})$ doesn't vanish identically on the fiber $X(\ul{v}),$ which is exactly the second condition. 

    For the first condition, use the fact that $T_m(\cF)$ is a regular sequence by lemmas \ref{lem:Taylor-strong} and \ref{lem:str-prime}. By lemma \ref{reg-to-surj} this implies that the corresponding polynomial map is surjective. In particular, $T_m(\cF_h)\restriction_{Z_m(\cF_{<h})}$ is surjective. By lemma \ref{lem:gen-str}, this implies that the first condition holds on a dense open subset of $Z_m(\cF_{<h})$. 
\end{proof}

\begin{corollary}
    Let $h\ge 2$ and suppose that theorem \ref{thm:dense} holds for towers of height $h-1$. Then it holds for towers of height $h$.
\end{corollary} 

\begin{proof}
    Let $A,B$ be such that theorem \ref{thm:dense} holds with these constants for all towers of degree $\le d$ and height $< h$. Apply the previous proposition with either 
    \begin{align*}
    m &= d2^d(A(s_h+\phi_2+\ldots+\phi_d)^B+s_h) \textnormal{ or} \\
    m &= d2^d(As_h^B+s_h).
    \end{align*}
    
     By claim \ref{str-shuffle}, there exist $A',B'$ such that if $\cF$ is absolutely $(A',B',1)$-strong then it is absolutely $(Am^{dB},B,1)$-strong. The tower $T_m(\cF_{<h})$ has degree $(d_1,\ldots,d_{h-1})$ and is $(A,B,1)$-strong by lemma \ref{lem:Taylor-strong} so by our assumption we can choose $\ul{v}\in W\subseteq Z_m(\cF_{<h})$ to be a $K$-point. By  corollaries \ref{Brauer-base} or \ref{base-num} we get that there exists $ a\in K^m$ with $a\cdot\ul{v} \in Z(\cF_h) \cap U$. Since $\ul{v} \in Z_m(\cF_{<h})$, we have obtained a $K$-point  $a\cdot\ul{v} \in Z(\cF) \cap U$ as desired. 
\end{proof}

\section{Properties of multi-linear forms}

We now get to work in earnest, building towards our main technical result -- theorem \ref{thm:reg}. We begin by recalling the standard correspondence between polynomials and multi-linear forms.

\begin{definition}\label{def:polar}
    For a homogeneous polynomial $f$  of degree $d$, its \emph{polarization} is the multi-linear form  $\Tilde{f}:V^d\to K$ given by $\Tilde{f}(h_1,\ldots,h_d) = f^{(1,\ldots,1)}(h_1,\ldots,h_d)$.
\end{definition}

Note that this correspondence preserves strength.

\begin{lemma}\label{lem:str-polar}
    We have 
    \[
    2^{-d} \cdot \str(\Tilde{f}) \le \str(f) \le \str(\Tilde{f}).
    \]
\end{lemma}

\begin{proof}
    Suppose first that $\str(\Tilde{f}) = r$ and write $\Tilde{f} = \sum_{i=1}^r p_i q_i$, where $p_i,q_i$ are forms of degree $<d$ on $V^d$. Plugging in $h_1=\ldots=h_d = x$, we get 
    \[
    d! f(x) = \Tilde{f}(x,\ldots,x) = \sum_{i=1}^r p_i(x,\ldots,x) q_i(x,\ldots,x),
    \]
    where the equality on the left follows from lemma \ref{lem:Taylor-identities}. This implies 
    \[
    \str(f) = \str(d! f) \le r = \str(\Tilde{f}),
    \]
    proving the inequality on the right hand side.

    Now suppose $\str(f) = r$ and write $f = \sum_{i=1}^r p_i q_i$, where $p_i,q_i$ are forms of positive degrees $d_i,d-d_i$.  Comparing the corresponding terms of the Taylor expansion of $f(h_1+\ldots+h_d)$ and $\sum_{i=1}^r p_i(h_1+\ldots+h_d)\cdot q_i(h_1+\ldots+h_d)$, we get 
    \[
    \Tilde{f}(h_1,\ldots,h_d) = \sum_{i=1}^r \sum_{J\in \binom{[d]}{d_i}}  \Tilde{p}_i((h_j)_{j\in J}) \cdot \Tilde{q}_i((h_j)_{j\in [d]\setminus J}).  
    \]
    This proves the inequality $\str(\Tilde{f}) \le 2^d\cdot \str(f)$.
\end{proof}

Applying the correspondence of definition \ref{def:polar} to a tower of homogeneous polynomials gives rise to a multi-linear tower. 

\begin{definition}
    A \emph{multi-linear tower} $\cF$ on $V^{[d]}$ is a tower where for each $i\in [h]$ there is some $I_i\subset [d]$ such that $(f_{i,j})_{j\in [s_h]}$ are all multi-linear forms on $V^{I_i}$. For a  tower $\cF$ with $d = \max_i d_i$, its polarization $\Tilde{\cF}$ is a multi-linear tower on $V^{[d]}$. For every $i\in [h]$ and $I\in \binom{[d]}{d_i}$ we get a layer $\Tilde{\cF}_i^I$ whose forms are $\Tilde{f}_{i,j}^I:V^I\to K$. The layers of $\Tilde{\cF}$  are ordered consecutively in the different $i$ and arbitrarily in the $I$. Given $i,I$ we write $\Tilde{\cF}_{<(i,I)}$ for the tower consisting of $\Tilde{\cF}_{i'}^{I'}$ for $i'<i$ or $i' = i$ and $I'\neq I$. 
\end{definition}

The correspondence for towers also preserves relative strength.

\begin{lemma}\label{lem:rel-str-polar}
   For every $i\in [h] $ and $I\in \binom{[d]}{d_i}$ we have 
    \[
    2^{-d} \cdot \str_{\Tilde{\cF}_{<(i,I)}} (\Tilde{\cF}_i^I) \le \str_{\cF_{<i}} (\cF_i) \le \str_{\Tilde{\cF}_{<(i,I)}} (\Tilde{\cF}_i^I).
    \]
\end{lemma}

\begin{proof}
        The proof of the left inequality is identical to lemma \ref{lem:str-polar}. 
        
        For the right inequality, suppose that $f = \sum_{j=1}^{s_i} \sum_{I\in \binom{[d]}{d_i}} a_j^I \Tilde{f}^I_{i,j}$ is a nontrivial linear combination with $\str_{\Tilde{\cF}_{<(i,I)}} (f) = r$. Then there exists some $I_0\in \binom{[d]}{d_i}$ such that $a_j^{I_0}$ are not all zero.  Plugging in $h_\ell = x$ for $\ell \in I_0$ and $h_\ell = 0$ for $\ell\notin I_0$, we get $\str_{\cF_{<i}} (d_i! \sum_{j=1}^{s_i} a_j^{I_0}f_{i,j}) \le r$, which proves the claim.  
\end{proof}

Due to this equivalence between relative rank for the original tower and for its polarization, it is sufficient to prove theorem \ref{thm:reg} for multi-linear towers.

\begin{corollary}\label{cor:ml-reduction}
    Suppose that theorem \ref{thm:reg} holds for multi-linear towers. Then it holds for all towers.
\end{corollary}

\begin{proof}
    Apply the previous lemma to both $\cF$ and $\cF\otimes_K \ol{K}$.
\end{proof}

In light of the above, we henceforth restrict our attention to multi-linear towers. In this setting there is a very simple geometric condition which ensures density of rational points.

\begin{proposition}\label{dense-ml}
    If $K$ is an infinite field and $\ul{F} = (F^I_i)$ is a collection of multi-linear forms on $V^{[d]}$ which cut out an irreducible complete intersection, then their rational zeros are Zariski dense.
\end{proposition}

\begin{proof}

    By induction on $d.$ The base case $d=1$ is immediate. For the inductive step, we need to show that $Y = \overline{Z(K)} \subset Z$ has full dimension. Consider the projection of $Z$ on $V^{[d-1]}$, denoted $Z'$. This is the image of an irreducible set, hence is irreducible itself. The set $Z'$ is the zero locus of $(F^I_i)_{I\subset [d-1]}$, so it is also a complete intersection. By inductive hypothesis, $Z'(K)$ is dense in $Z'$. For each point $x\in Z'(K),$ the fiber $Y(x)$ is a closed subset of the linear subspace $Z(x)\subset V_d$ which contains all the $K$-points, so $Y(x) = Z(x).$ This implies that $  \dim Y \ge \dim Z' + \dim V_d - r = \dim Z$, where $r$ is the number of equations involving the $d$-th coordinate.
   
\end{proof}
    
For multi-linear forms, strength admits a natural refinement to partition rank.

\begin{definition}
    Let $F:V^{[d]}\to K $ be a multi-linear form. We say that $F$ has partition rank one if $F(x) = G(x_I)\cdot H(x_{I^c})$ for a proper nonempty subset $I\subset [d]$ and $G,H$ multi-linear forms. The \emph{partition rank} of $F$ is 
    \[
    \prk(F) = \min\{r: F = F_1+\ldots+F_r\},
    \]
    where each $F_i:V^{[d]}\to K$ is multi-linear of partition rank one. If $I\subset [d]$, $F:V^I\to K$ is multi-linear and $G_i:V^{I_i}\to K$ are multi-linear then the \emph{relative partition rank} of $F$ on $\ul{G}$ is 
    \[
    \prk_{\ul{G}}(F) = \min\{\prk(F+H) | H:V^I\to K \textnormal{ is multi-linear and } H\in (\ul{G})\}
    \]
\end{definition}

In a similar way, Birch rank may be refined to geometric rank.

\begin{definition}
    The \emph{geometric rank} of $F:V^I\to K$ is 
    \[
    \grk(F) = \codim (x:F(x_{I\setminus\{i_0\}},\cdot)\equiv 0),
    \]
    where $i_0\in I$ is any element. Note that this quantity does not depend on the choice of $i_0.$
    The \emph{relative geometric rank} of $F$ on $\ul{G}$ is 
    \[
    \grk_{\ul{G}} (F) = \min_{i_0\in I} \codim_{Z(\ul{G})} (x\in Z(\ul{G}): F(x_{I\setminus\{i_0\}},\cdot),\ul{G}^{i_0}(x_{[d]\setminus \{i_0\}},\cdot ) \textnormal{ are linearly dependent} ),
    \]
    where $\ul{G}^{i_0}$ are the multi-linear forms in $\ul{G}$ which involve $i_0$.
\end{definition}

Using these definitions, we can refine the definitions of strong/regular towers to the multi-linear setting in the obvious way. As in the non-relative setting, low relative partition rank implies low relative geometric rank:

\begin{claim}\label{cl:rk-sing}
   We have $\grk_{\ul{G}}(F) \le \overline{\prk}_{\ul{G}} (F).$
\end{claim}

\begin{proof}
    If $F = \sum_{i=1}^r P_iQ_i \mod \ul{G}$ then $F(x,\cdot) \in \sp(G(x,\cdot))$ whenever $P_i(x) = 0$.
\end{proof}

The main result we will use is:

\begin{proposition}\cite[Theorem 3.5]{L-small}\label{prop:str-reg}
    Given $A,B$, there exist $C,D$ such that if $\cF$ is a multi-linear tower which is absolutely $(C,D,r)$-strong (in the sense of partition rank) then it is $(A,B,r)$-regular (in the sense of geometric rank).
\end{proposition}

A key role will also be played by the regularity of towers obtained by fixing subsets of the coordinates.

\begin{definition}
    Suppose  $\cF$ is a multi-linear tower on $V^{[d]}$ and $I\subset [d]$. We denote by $\cF^I$ the multi-linear tower on $V^I$ consisting of layers which only involve coordinates from $I$. Given $x\in Z(\cF^I)$, write $\cF(x)$ for the multi-linear tower on $V^{[d]\setminus I}$ obtained by fixing the $I$ coordinates to equal $x$ in layers which involve some coordinate outside $I$. In other words, these are the equations for the fiber of the projection map $Z(\cF) \to Z(\cF^I)$ over $x$.  
\end{definition}

\begin{lemma}\label{lem:fixed-reg}
    Given $A,B$, there exist $C,D$ such that if $\cF$ is absolutely $(C,D,r)$-strong then for $x\in Z(\cF^I)$ outside a set of codimension $\ge r$, the tower $\cF(x)$ is absolutely $(A,B,r)$-strong. 
\end{lemma}

\begin{proof}
    This is proved by induction on $h,$ where the base case $h = 0$ is trivial. Suppose the statement holds for height $<h$ and let $\cF$ be a multi-linear tower of height $h$ on $V^{[d]}$. Note that $\cF_{<h}$ is absolutely $(C,D,r+s_h)$-strong whenever $\cF$ is $(C,D,r)$-strong.

    There are two cases to handle. The first case is when $I_h \subset I$. In this case $\cF(x) = \cF_{<h}(x)$ for each $x$  and by inductive hypothesis the tower $\cF(x)$ is $(A,B,r)$-regular for $x\in Z(\cF_{<h}^I)$ outside a bad set $E$ of codimension $r+s_h$. Thus,
    \[
    \codim_{Z(\cF^I)} E\cap Z(\cF^I) \ge \dim Z(\cF^I) - \dim E  = \dim Z(\cF_{<h}^I) - s_h - \dim E \ge r,
    \] 
    where we used the fact that $Z(\cF_I)$ is a complete intersection by lemma \ref{lem:str-prime}. 

    The second case is when $I_h \not\subset I$. In this case, $Z(\cF_{<h}^I) = Z(\cF^I)$. By inductive hypothesis, $\cF_{<h}(x)$ is absolutely $(A,B,r+s_h)$-strong for $x\in Z(\cF^I)$ outside a set of codimension $r$, so we need to show that $\ol{\str}_{\cF_{<h}(x)} (\cF_h(x)) > A(s_h+r)^B$ for $x\in Z(\cF^I)$ outside a set of codimension $r$.  Suppose, without loss of generality, that $d\in I_h\setminus I$ and let  
    \[
    S = (x\in Z(\cF_{<h}): \cF(x_{[d-1]},\cdot) \textnormal{ are linearly dependent}).
    \]
    
    For any fixed $\eta$, the set
    $E_\eta = (x\in Z(\cF_I): \codim_{Z(\cF_{<h})(x)} S(x) \le \dim Z(\cF_I) + \eta)$ satisfies
    \begin{equation}\label{eq:fiber-str}
        \alpha(s+r)^\beta < \codim_{Z(\cF_{<h})} S \le \codim_{Z(\cF_I)} E_\eta +\eta,
    \end{equation}
    where the inequality on the left follows from proposition \ref{prop:str-reg} and the one on the right is a general property for dimension of fibers.
    By claim \ref{cl:rk-sing} the set of $x\in Z(\cF_I)$ with $ \ol{\str}_{\cF_{<h}(x)} (\cF_h(x)) \le \eta$ is contained in $E_\eta$ so it's enough to show that $E_{A(s_h+r)^B+1}$ has codimension $\ge r$. By inequality \eqref{eq:fiber-str} this holds for appropriately chosen $\alpha,\beta$, and hence for the corresponding choice of $C,D$. 
\end{proof}

We often rely on a stronger version of lemma \ref{lem:fixed-reg}, namely that $\cF\cup \cF(\ul{y})$ is regular for generic choices of $\ul{y}\in Z(\cF_I)^m$. To prove this, we first introduce two "cloning" operations for multi-linear equations. Given a multi-linear form $F:V^{[d]}\to K$, we can set $V_{d+1}=V_d$ and define forms $F_1,F_2:V^{[d+1]}\to K$ by $F_1(x_1,\ldots,x_{d+1}) = F(x_1,\ldots,x_d)$ and $F_2(x_1,\ldots,x_{d+1}) = F(x_1,\ldots,x_{d-1},x_{d+1})$. These are multi-linear in the $[d]$ coordinates and $[d-1]\cup\{d+1\}$ coordinates, respectively, which is an example of \emph{external} cloning. Alternatively, we can set $U_i = V_i$ for $i<d$ and $U_d = V_d^2$ and then define $F_1,F_2:U^{[d]}\to K$ by $F_i(x_1,\ldots,x_d) = F(x_1,\ldots,x_d(i))$. These are both multi-linear in the $[d]$ coordinates, which is an example of \emph{internal} cloning. We now give the general definition. 

\begin{definition}
    Given a multi-linear form $F:V^{[d]}\to K$ and a subset $I\subset [d]$, we can produce $m$ multi-linear forms  $F_i:V^{[d]\setminus I}\times \prod_{i=1}^m V^I \to K$ by $F_i(x,y) = F(x,y(i))$. When the $F_i$ are considered as multi-linear forms on different subsets of $[d+(m-1)|I|]$, we call this operation \emph{external cloning}  of degree $m$ over $I$. When they are all considered as multi-linear forms on $[d]$, we call this \emph{internal cloning} of degree $m$ over $I$. These operations naturally extend to multi-linear towers. External cloning produces a tower with at most $m$ times as many layers and internal cloning produces a tower with the same number of layers. 
\end{definition}

These cloning operations preserve regularity.

\begin{lemma}\label{lem:clone-str}
Let $\cF$ be a multi-linear tower on $V^{[d]}$ and let $\cF', \cF''$ be obtained by external, internal cloning of degree $m$ over $I$, respectively. 
\begin{enumerate}
    \item If $\cF$ is $(A,B,r)$-strong then so is $\cF'$.
    \item If $\cF$ is $(Am^B,B,r)$-strong then $\cF''$ is $(A,B,r)$-strong.
\end{enumerate}    
\end{lemma}

\begin{proof}
    If $\cF = (\cF_i)_{i\in [h]}$ then $\cF' = (\cF_{i,j})_{i\in [h], j\in [r]}$ with the layers arranged in lexicographic order. The first claim follows from $\str_{\cF'_{<(i,j)}} (\cF'_{i,j}) \ge \str_{\cF_{<i}} (\cF_i)$ for all $i,j$. Indeed, suppose that
    \[
    t = \str_{\cF'_{<(i,j)}} \left( \sum_{k=1}^{s_i} a_k F_{i,k}(x,y_j) \right)
    \]
    for a nontrivial linear combination of the forms in $\cF'_{i,j}$. Plugging in $y_j = y$ and $y_\ell = 0 $ for all $\ell \neq j$ yields 
    \[
    \str_{\cF_{<i}} (\cF_i) \le \str_{\cF_{<i}} \left( \sum_{k=1}^{s_i} a_k F_{i,k}(x,y) \right) \le t \le \str_{\cF'_{<(i,j)}} (\cF'_{i,j}). 
    \]

    The proof of the second claim is almost identical.
\end{proof}

We can finally prove:

\begin{lemma}\label{lem:reg-der}
    Given $A,B$ there exist $C,D$ such that if $\cF$ is a multi-linear tower which is absolutely $(C,D,r+m)$-strong then for $\ul{y}\in Z(\cF_I)^m$ outside a subset of codimension $r$, the tower $\cF\cup\cF(\ul{y})$ is absolutely $(A,B,r)$-strong.
\end{lemma}

\begin{proof}
    Let $\cF'$ be obtained from $\cF$ by externally doubling the $I$ coordinates. Then construct $\cF''$ from $\cF'$ by internally cloning the second copy of the $I$ coordinates $m$ times. Note that fixing those $m$ copies to equal $\ul{y} = (y_1,\ldots,y_m)$ gives us $\cF''(\ul{y}) = \cF\cup\cF(\ul{y})$. By lemma \ref{lem:clone-str}, the tower $\cF''$ is absolutely $(A,B,r)$-strong as soon as $\cF$ is absolutely $(C,D,r+m)$-strong. Now apply lemma \ref{lem:fixed-reg} to complete the proof.  
\end{proof}

\section{Proof of theorem \ref{thm:reg}}

We begin with a few reductions. First, we reproduce a useful argument which we learned from A. Polishchuk.

\begin{lemma}\label{lem:Gal-des}
   Suppose that $ \overline{\str}_{\cF_{<h}} (\cF_h) = t.$ Then there exist $a_1,\ldots,a_{s_h} \in K$ not all zero such that $\overline{\str}_{\cF_{<h}} (a_1F_{h,1}+\ldots+a_{s_h}F_{h,s_h}) \le s_h\cdot t$. 
\end{lemma}

\begin{proof}
    Suppose that $F = b_1F_{h,1}+\ldots+b_{s_h}F_{h,s_h}$ is a nontrivial $\ol{K}$-linear combination of the forms in $\cF_h$ with $\overline{\str}_{\cF_{<h}} (F)  = t. $ Let $ U = \sp_{\bar K}(\sigma F) \subset \sp_{\bar K}(\cF_h),$ where $\sigma F$ runs over all Galois conjugates. By Galois descent, $ \bar U = \sp_{\bar K} (G_1,\ldots,G_\ell) $ for some $\ell \le s_h$ and $G_i\in \sp_K(\cF_h)$. Choose some $0\neq G\in \sp_K (G_1,\ldots,G_\ell) \subset \sp_K(\cF_h)$. We may write $G = \alpha_1 \sigma_1 F+\ldots \alpha_\ell \sigma_\ell F$ for some $\alpha_i\in \bar K$ and Galois conjugates $\sigma_i F.$ Since $\sigma (\cF_{<h}) = \cF_{<h}$, we have $\overline{\str}_{\cF_{<h}} (\sigma_i F) = \overline{\str}_{\cF_{<h}} (F)  = t.$ By subadditivity,
    \[
    \overline{\str}_{\cF_{<h}} (G) = \overline{\str}_{\cF_{<h}} (\sum_{j=1}^\ell \alpha_j \sigma_j F) \le \sum_{j=1}^\ell \overline{\str}_{\cF_{<h}} (F) \le  s_h\cdot t.
    \]
\end{proof}

The proof of theorem  \ref{thm:reg} will be by induction on the degree sequence $\ul{d}$, where $\ul{e} < \ul{d}$ if $\ul{e}$ can be obtained from $\ul{d}$ by replacing some entry by any number of lower ones. We note another useful reduction.

\begin{lemma}\label{lem:non-max}
    Let $\ul{d} = (d_1,\ldots,d_h)$ be a degree sequence with $d_i > d_h$ for some $i < h.$ If theorem \ref{thm:reg} holds for degrees $<\ul{d}$ then it holds for degree $\ul{d}$.
\end{lemma}

\begin{proof}
    Let $\cF$ be a tower of degree $\ul{d}$ which is $(C,D,r)$-strong. Then the tower $\cF_{<h}$ is $(C,D,r+s_h)$-strong, hence absolutely $(A,B,r+s_h)$-strong by theorem \ref{thm:reg}. The theorem also implies that the tower $(\cF_j)_{j\neq i}$ is absolutely $(A,B,r)$-strong, so in particular
    \[
    \ol{\str}_{\cF_{<h}} (\cF_h) = \ol{\str}_{(\cF_j)_{j\neq i,h}} (\cF_h)  > A(r+s_h)^B.
    \]
    Therefore $\cF$ is absolutely $(A,B,r)$-regular as claimed. 
\end{proof}
    
The inductive proof will establish the validity of the following statement for all $\ul{e},d$:

 \begin{description}[align=right,labelwidth=1.5cm,leftmargin=!]
\item[$\Sigma(\ul{e},d)$] There exist quantities $A(\ul{e},d),B(\ul{e},d)$ such that if $\cG$ is a multi-linear tower of degree $\ul{e}$ which is absolutely $(A,B,r)$-strong and $F$ is a multi-linear form of degree $d$ satisfying $\overline{\prk}_{\cG}(F) = r$, then $\prk_{\cG}(F) \le Ar^B$. 
\end{description}

We introduce a refinement of the partial order on degree sequences.

\begin{definition}
    For degree sequences $\ul{d},\ul{e}$ write $\ul{e} <^{\textnormal{top}} \ul{d}$ if:
\begin{enumerate}
    \item $\max_i e_i \le d$ and also 
    \item  $|\{i:e_i = d\}| < |\{i:d_i = d\}|$.
\end{enumerate}
\end{definition}

We will prove that $\Sigma(\ul{e},d)$ holds for \emph{every} $\ul{e} <^{\textnormal{top}} \ul{d}$, which immediately implies theorem \ref{thm:reg} for degree $\ul{d}$ by taking $\ul{e} = (d_1,\ldots,d_{h-1})$.

The proof will be by induction on the partitions appearing in the partition rank decomposition. For expository purposes we sketch the argument in a simple case.

\subsection*{Proof sketch} Suppose $\cG$ is a tower of bilinear forms and that $F:V^{[3]}\to K$ is trilinear with $\overline{\prk}_{\cG}(F) = 2$, given by $F(x,y,z) = \alpha (x) A(y,z) + \beta(z)B(x,y) \mod \cG$. The forms $\alpha,A,\beta,B$ have coefficients in $\overline{K}$ and our goal is to replace them by a small collection of forms with coefficients in $K$. Our assumption $\overline{\prk}_{\cG}(F) = 2$ implies that $\beta\neq 0$ and therefore we can choose a $K$-point $z_0\in V_3$ with $\beta(z_0) \neq 0$. Plugging this in yields $F(x,y,z_0) = \alpha(x) A(y,z_0) + \beta(z_0)B(x,y) \mod \cG(z_0)$, implying $B\in (\alpha,F(\cdot,z_0),\cG(z_0))$. Therefore $F = \alpha (x) A'(y,z) \mod \cG,\cG(z_0), F(\cdot,z_0)$. This last ideal is generated by forms with coefficients in $K$. By lemma \ref{lem:reg-der}, $\cG\cup\cG(z_0)$ is relatively strong for $z_0$ in a dense open subset of $V_3$. Applying proposition \ref{prop:rel-reg}, we can replace $F(\cdot,z_0)$ by a small collection $\cH$ which is relatively strong $\mod \cG\cup\cG(z_0)$. Now we have $\overline{\prk}_{\cG\cup \cG(z_0) \cup \cH}(F) \le 1$ but with only the partition $(\{1\},\{2,3\})$ appearing! Now we can apply an appropriate inductive hypothesis to get $\prk_{\cG\cup \cG(z_0) \cup \cH} (F) \le C $. This implies $\prk_{\cG\cup \cG(z_0)} (F) \le C+|\cH| \le C'$. Finally, we will see in lemma \ref{lem:gluing-deriv} that the fact that this is true for many choices of $K$-points $z_0$ implies $\prk_{\cG}(F) \le C''$. 

In general, to carry out the induction on partitions of $[d]$ we make the following definition. 

\begin{definition}
    Let $\cI$ be a collection of proper subsets of $[d],$ all containing the element $1.$ The $\cI$-rank, written $\prk^\cI(F),$ is the minimal $r$ such that $ F = \sum_{i=1}^r G_i(x_{I_i}) H_i(x_{[d]\setminus I_i})$,
    where $G_i,H_i$ are multi-linear forms and $I_i\in \cI$ for all $i\in [r].$ For example, taking $\cI = \{\{1\}\} \cup \{[d]\setminus \{j\}\}_{j\neq 1}$ yields the slice rank. The relative $\cI$-partition rank is defined analogously: $\prk^\cI_\cG(F) = \min \{ \prk^\cI(F+G): G\in I(\cG)_{[d]} \}.$ 
\end{definition}

We partially order these collections as follows: If $\cI'$ is obtained from $\cI$ by replacing some of the sets by any number (possibly zero) of their proper subsets which contain $1$ then $\cI' < \cI$. 

\begin{example}
The proper subsets of $[3]$ containing the element $1$ are $\{1\}, \{1,2\}, \{1,3\}$. In this case the nonempty collections of subsets are ordered as follows: 

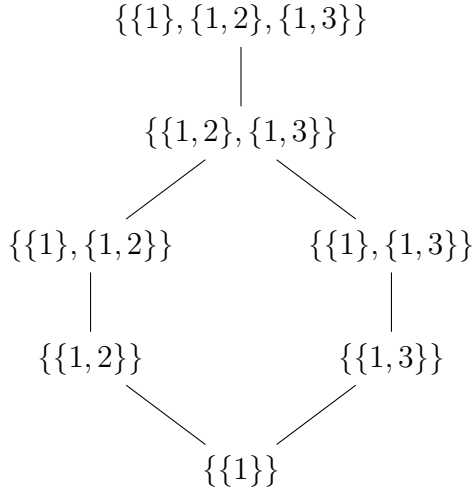
\begin{center}

\begin{tikzpicture}[
    node distance=1.5cm,
    level distance=1.5cm
]

    \node (A) at (0, 3) {$\{\{1\},\{1,2\},\{1,3\}\}$};

    \node (B) at (0,1.5) {$\{\{1,2\},\{1,3\}\}$};
    
    \node (C1) at (-2, 0) {$\{\{1\},\{1,2\}\}$};
    \node (C2) at (2, 0) {$\{\{1\},\{1,3\}\}$};

    \node (D1) at (-2, -1.5) {$\{\{1,2\}\}$};
    \node (D2) at (2, -1.5) {$\{\{1,3\}\}$};

    \node (E) at (0, -3) {$\{\{1\}\}$};

    \draw (A) edge (B);
    \draw (B) edge (C1);
    \draw (B) edge (C2);
    \draw (C1) edge (D1);
    \draw (C2) edge (D2);
    \draw (D1) edge (E);
    \draw (D2) edge (E);

\end{tikzpicture}

\end{center}

\end{example}

With the above definition, we can now refine the statement $\Sigma(\ul{e},d)$.

 \begin{description}[align=right,labelwidth=1.5cm,leftmargin=!]
\item[$\Sigma^\cI(\ul{e},d)$] There exist quantities $A(\ul{e},d),B(\ul{e},d)$ such that if $\cG$ is a multi-linear tower of degree $\ul{e}$ which is absolutely $(A,B,r)$-strong and $F$ is a multi-linear form of degree $d$ satisfying $\ol{\prk}^\cI_{\cG}(F) \le r$, then $\prk^\cI_{\cG}(F) \le Ar^B$. 
\end{description}

Having set up our inductive framework and the appropriate reductions, our task is now to prove:

\begin{proposition}\label{prop:I-rank}
    Suppose theorem \ref{thm:reg} holds for multi-linear towers of degree $<\ul{d}$ and that $d = d_h = \max_i d_i$. Then $\Sigma^\cI(\ul{e},d)$ holds for every $\ul{e} <^{\textnormal{top}} \ul{d}$ and every collection $\cI$ of proper subsets of $[d]$, all containing the element $1$.
\end{proposition}

Indeed, we have:

\begin{lemma}
    Proposition \ref{prop:I-rank} implies theorem \ref{thm:reg}.  
\end{lemma}

\begin{proof}
    The proof is by induction on $\ul{d}$. When $\ul{d} = (1,\ldots,1)$ there is nothing to prove. Supposing that theorem \ref{thm:reg} holds for all multi-linear towers of degree $<\ul{d}$, we prove that it holds also in degree $\ul{d}$. If $d_h\neq \max_i d_i$ then we are done by lemma \ref{lem:non-max}, so assume that $d = d_h = \max_i d_i$. If $\cF$ is $(C,D,r)$-strong, then $\cF_{<h}$ is $(C,D,r+s_h)$-strong and hence is absolutely $(A,B,r+s_h)$-strong by theorem \ref{thm:reg}.  It remains to show that  $\ol{\str}_{\cF_{<h}} (\cF_h) > A(r+s_h)^B$. 
    
    Suppose, to get a contradiction, that this is not the case. By lemma \ref{lem:Gal-des} there are  $a_1,\ldots,a_{s_h}\in K$ not all zero such that $F = \sum_{j=1}^{s_h} a_j F_{h,j}$ satisfies $\ol{\str}_{\cF_{<h}} (F) \le A(r+s_h)^{B+1}$. By proposition \ref{prop:I-rank} applied with $\ul{e} = (d_1,\ldots,d_{h-1})$ and $\cI = \{I\subset [d]: 1\in I\}$ we deduce that $\str_{\cF_{<h}} (F) \le \alpha(r+s_h)^\beta$. For large $C,D$ this  contradicts the fact that $\cF$ is $(C,D,r)$-strong. 
\end{proof}

Sufficiently motivated, we now turn to proving proposition \ref{prop:I-rank}.   The following lemma plays a crucial role.

\begin{lemma}[Gluing low strength representations]\label{lem:gluing-deriv}
    Suppose theorem \ref{thm:reg} holds for multi-linear towers of degree $ <^{\textnormal{top}} \ul{d}$. Then there exist constants $A(\ul{e},d),B(\ul{e},d)$ such that the following holds. Let $\cG$ be a multi-linear tower of degree $\ul{e} <^{\textnormal{top}} \ul{d}$  which is absolutely $(A,B,r+m)$-strong and let $F$ be a multi-linear form of degree $d$ such that $\prk_{\cG\cup \cG(\ul{y})} (F) \le r$ for all $K$-points  $\ul{y}=(y_1, \ldots, y_m)$ in a dense open subset of $Z(\cG_I)^m.$ Then $\prk_{\cG}(F) \le Ar^B$.
\end{lemma}

We assume this lemma for now and defer its proof to later. 

 

\begin{proof}[Proof of proposition \ref{prop:I-rank}]
    
    The proof is by induction on $\cI$. The base of the induction is when $\cI$ is empty. If  $\overline{\prk}^\emptyset_\cG(F)$ is finite then $F\in (\cG)$ so $\prk_\cG(F) = 0.$ 
    
    To illustrate the argument, we explicitly work out the first non-trivial case $\cI = \{1\}$. In this case, our assumption is
\begin{equation}\label{eq:slice-rk}
    F = \sum_{i=1}^r \alpha_i (x_1)A_i(y) \mod \cG,
\end{equation}
where $y\in V^{[2,d]}.$ Shortening the sum if necessary,
we may assume that $A_i \mod \cG$ are linearly independent. Let $Z = Z(\cG_{[2,d]}) \subset V^{[2,d]}.$ We claim that the image of $\ul{A}:Z \to \bar K^r$ is not contained in any proper subspace. This is because otherwise some non-trivial linear combination $A = c_1A_1+\ldots+c_rA_r$ vanishes identically on $Z$, implying  $A\in (\cG_{[2,d]})$ since $\cG_{[2,d]}$ is prime by lemma \ref{lem:str-prime}. This would contradict the linear independence of $A_i \mod \cG.$

Consider the map $\ul{A}^r:Z^r\to M_{r\times r}(\bar K).$ We've just seen that the image of $\ul{A}$ is not contained in any proper subspace and therefore  the image of $\ul{A}^r$ contains an invertible matrix. It follows that for a dense open set of $\ul{y} = (y_1,\ldots,y_r)\in Z^r,$ the matrix $(A_i(y_j))$ is invertible. Proposition \ref{dense-ml} implies that we may choose $\ul{y}$ to be a $K$-point with this property. Plugging into equation \eqref{eq:slice-rk} yields
\[
F(x_1,y_j) = \sum_{i=1}^r \alpha_i (x_1)A_i(y_j) \mod \cG(y_j) \ \textnormal{for all } 1\le j\le r.
\]
Multiplying this system of equations by the inverse matrix $(A_i(y_j))^{-1},$ we deduce that $\alpha_i \in \sp_{\bar K} (\cG(\ul y),F(\cdot,\ul y))$ and therefore $F\in (\cG\cup \cG(\ul y)\cup F(\cdot,\ul y)).$ Note that $\cG,\cG(\ul y),F(\cdot,\ul y)$ are linear forms with \emph{coefficients in $K$} so $\prk_{\cG\cup \cG(\ul y)}(F) \le r$.\footnote{The fact that $F$ is in the ideal means that a certain set of linear equations has a solution in $\overline{K}$. These equations have coefficients in $K$ and thus there is a solution in $K$ as well.} The proof is completed by invoking lemma \ref{lem:gluing-deriv} to obtain $\prk_\cG(F) \le Ar^B$. 

Now we prove the general inductive step. Let $\cI$ be a collection of subsets containing the element $1,$ such that $\overline{\prk}^\cI_\cG(F) \le t.$ Let $I\in \cI$ be a \emph{maximal} element with respect to inclusion. By assumption, we can write 
\begin{equation}\label{eq:gen-case}
    F = \sum_{i=1}^r \alpha_i(x)A_i(y) + \sum_{i=1}^r \beta_i(x_{I_i})B_i(x_{[d]\setminus I_i}) \mod \cG,
\end{equation}
where $y\in V^{[d]\setminus I},\ x\in V^I$ and $I\neq I_i\in\cI$ for all $i.$  By shortening the sum if necessary, we may assume that the $A_i \mod \cG$ are linearly independent. 

As in the argument above, this implies that for generic $K$-points $\ul{y}\in Z(\cG_{[d]\setminus I})^r,$ the matrix $(A_i(y_j))$ is invertible. Plugging this into equation \eqref{eq:gen-case},
\[
F(\cdot,y_j) = \sum_{i=1}^r \alpha_i(x)A_i(y_j) + \sum_{i=1}^r \beta_i(x_{I_i\cap I},y_{[d]\setminus I})B_i(x_{I\setminus I_i},y_{[d]\setminus I}) \mod \cG(y_j)\; \textnormal{for all } 1\le j\le r.
\]
Note that the second sum on the right hand side has partition rank $\le r$. Otherwise there must be some $i\in [r]$ with $I\subseteq I_i$ or $I\subseteq [d]\setminus I_i$. The first inclusion is impossible because $I$ is maximal in $\cI$ and the second inclusion is impossible because $1\in I\cap I_i$.
Multiplying by the inverse matrix yields $\overline{\prk}_{\cG(\ul{y}) \cup F(\ul{y},\cdot)}(\alpha_i) \le r^2 $ for all $i\in [r].$ Then 
\[
\overline{\prk}^{\cI'}_{\cG\cup \cG(\ul{y}) \cup F(\ul{y},\cdot)}(F) \le r^3,
\]
where $\cI' = \{I\cap J: I\neq J\in \cI\} < \cI$.

For any fixed $K$-point $\ul{y},$ we can apply proposition \ref{prop:rel-reg} to replace $F(\ul{y},\cdot)$ by a tower $\cH_{\ul{y}}$ of $K$-forms of size $\le \alpha r^\beta$ such that $\cH_{\ul{y}} \mod \cG\cup \cG(\ul{y})$ is $(C,D,r)$-strong and also $(F(\ul{y},\cdot)) \subset (\cG\cup \cG(\ul{y})\cup \cH_{\ul{y}})$.  For generic $\ul{y}\in Z^r$, The tower $\cG\cup \cG(\ul{y})$ is absolutely $(C,D,r)$-strong by lemma \ref{lem:reg-der}. Choosing $\ul{y}$ to be such a $K$-point, the tower $\cP_{\ul{y}} = \cG\cup \cG(\ul{y})\cup \cH_{\ul{y}}$ is $(C,D,r)$-strong. Our assumption that $\ul{e} <^{\textnormal{top}} \ul{d}$ implies that $\cP_{\ul{y}}$ has degree $\ul{d}'< \ul{d}$ and so by theorem \ref{thm:reg} it is absolutely $(A,B,r)$-strong. Recall that $\ol{\prk}^{\cI'}_{\cP_{\ul{y}}}(F) \le r^3$ by construction. 

Applying our inductive hypothesis that the statement $\Sigma^{\cI'}(\ul{d}',d)$ holds, we get that $\prk_{\cP_{\ul{y}}}(F) \le Ar^B$, so that $\prk_{\cG\cup \cG(\ul{y})} (F) \le Ar^B+|\cH| \le A'r^{B'}$. Applying lemma \ref{lem:gluing-deriv}, we deduce that $\prk_\cG (F) \le A''r^{B''}$, completing the inductive step. 
\end{proof}


\subsection{Proof of lemma \ref{lem:gluing-deriv}}

We now pay our dues by proving lemma \ref{lem:gluing-deriv}. From here on we assume that theorem \ref{thm:reg} holds for multi-linear towers of degree $ <^{\textnormal{top}} \ul{d}$. The proof will proceed by "gluing" the various low rank representations together, a process introduced in \cite{LZ-rel}. We begin with a useful consequence of proposition \ref{prop:rel-reg}.

\begin{lemma}\label{rel-rel-reg}
    Given $A,B$ there exist $C(\ul{e},d),D(\ul{e},d)$ such that the following holds. If $\cG$ is a multi-linear tower of degree $\ul{e}$ which is  $(C,D,r)$-strong and $H_1,\ldots,H_r$ are multi-linear forms of degrees $< d$ then there exists a tower $\cH$ of degree $(1,\ldots,d-1)$ such that:
    \begin{enumerate}
        \item $H_i\in (\cG\cup\cH)$ for all $i\in [r]$,
        \item $|\cH| \le Cr^d$ and 
        \item $\cG\cup\cH$ is $(A,B,1)$-strong.  
    \end{enumerate}
        
\end{lemma}

\begin{proof}
    By proposition \ref{prop:rel-reg}, there exists $\cH$ of size $< A'r^{B'}$ and degree $(1,\ldots,d-1)$ such that $H_i\in (\cG\cup\cH)$ for all $i\in [r]$ and $\cH \mod (\cG)$ is $(A,B,1)$-strong. By claim \ref{str-shuffle}, if $\cG$ is $(C,D,r)$-strong then it is also $(A,B,A'r^{B'})$-strong. By the same claim, in this case the tower $\cG\cup\cH$ is $(A,B,1)$-strong as desired. 
\end{proof}

\begin{lemma}\label{lem:null-str}
    There exist $C,D$ such that the following holds. Suppose that $\cG$ is a multi-linear tower of degree $\ul{e} <^{\textnormal{top}} \ul{d}$ which is $(C,D,r)$-strong and $F$ is a multi-linear form of degree $d$. If there are multi-linear forms $H_1,\ldots,H_r$ of degree $<d$ with $F\restriction_{Z(\cG,H_1,\ldots,H_r)} = 0$ then $\prk_{\cG}(F) \le Cr^D$.
\end{lemma}

\begin{proof}
    By lemma \ref{lem:str-prime}, there exist $A',B'$ such that if a multi-linear tower of degree $(e_1,\ldots,e_h,1,\ldots,d-1)$  is absolutely $(A',B',1)$-strong then it generates a prime ideal. By theorem \ref{thm:reg}, there exist $A,B$ such that it is enough for such a tower to be $(A,B,1)$-strong. Applying lemma \ref{rel-rel-reg}, the resulting tower $\cG\cup\cH$ is $(A,B,1)$-strong and $F\restriction_{Z(\cG\cup\cH)} = 0$. By the nullstellensatz,\footnote{Really the nullstellensatz implies that $F$ is in the ideal generated by multi-linear forms with $\overline{K}$-coefficients. But this just means that a certain set of linear equations has a solution in $\overline{K}$. These equations have coefficients in $K$ and thus there is a solution in $K$ as well.} we deduce that $F\in (\cG\cup\cH)$ and so $\prk_{\cG}(F) \le |\cH| \le Cr^D$. 
\end{proof}

We now state our "gluing" lemma, which allows us to combine information from low rank representation on different towers.

\begin{lemma}\label{lem:gluing-two}
    For any $\ul{e}<^{\textnormal{top}} \ul{d}$ there exist $A(\ul{e},d),B(\ul{e},d)$ such that the following holds. Suppose that $\cG\cup\cG_1\cup\cG_2$ is a multi-linear tower of degree $\ul{e}$ which is absolutely $(A,B,r)$-strong and that $F$ is a multi-linear form of degree $d$. If we have $\prk_{\cG\cup \cG_i} (F) \le r$ for $i = 1,2$ then $\prk_{\cG\cup(\cG_1\cup\cG_2)_{[d-1]}}(F) \le Ar^B$.
\end{lemma}

\begin{proof}
    By assumption, we have multi-linear forms $Q^{(i)}_1,\ldots,Q^{(i)}_r$  depending only on the first $d-1$ coordinates such that $F\restriction_{Z(\cG\cup \cG_i ,  Q^{(i)}_1,\ldots,Q^{(i)}_r)} = 0$ for $i = 1,2$. We claim that $F$ vanishes on  
    \[
    Z = Z(\cG\cup(\cG_1\cup\cG_2)_{[d-1]},Q^{(1)}_1,\ldots,Q^{(1)}_r, Q^{(2)}_1,\ldots,Q^{(2)}_r).
    \]
    Once this is established, lemma \ref{lem:null-str} completes the proof.

    To prove that $F$ vanishes on $Z$, it's enough to prove that it vanishes on a dense subset. Consider the variety $Z' = Z(\cG\cup(\cG_1\cup\cG_2)_{[d-1]})$. By lemma \ref{lem:fixed-reg}, the tower $Z'(x)$ is $(1,1,1)$-strong for $x\in Z'_{[d-1]}$ outside a set of codimension $\ge 2r+1$. This simply means that the linear forms $\cG(x),\cG_1(x),\cG_2(x)$ are linearly independent. By considering the projection $Z'\to Z'_{[d-1]}$, this property holds for $(x,y)\in Z'$ outside a set of codimension $\ge 2r+1$.

    In particular, it holds for a dense open subset $U\subset Z$. If $(x,y)\in U$ then we can write $y = y_1+y_2$ where $(x,y_i)\in Z(\cG\cup\cG_i)$. By assumption, $F(x,y_i) = 0$ and thus $F(x,y) = F(x,y_1)+F(x,y_2) = 0$ as claimed.  
\end{proof}

\begin{corollary}\label{cor:gluing-many}
    For any $\ul{e}<^{\textnormal{top}} \ul{d}$ and $k\le d$ there exist $A(\ul{e},d,k),B(\ul{e},d,k)$ such that the following holds. Suppose that $\cG\cup\cG_1\cup\ldots\cup\cG_{2^k}$ is a multi-linear tower of degree $\ul{e}$ which is absolutely $(A,B,r)$-regular and that $F$ is a multi-linear form of degree $d$. If for each $i\in [k]$ we have $\prk_{\cG\cup\cG_i} (F) \le r$ then $\prk_{\cG\cup (\cG_1\cup\ldots\cup\cG_{2^k})_{[d-k]}} (F) \le Ar^B$.
\end{corollary}

\begin{proof}
    The proof is by induction on $k$. Lemma \ref{lem:gluing-two} is the base case $ k =1.$ Suppose that the statement holds for $k' \le k$ and that we wish to prove it for $k+1$. By applying the inductive hypothesis for $k$ we get that $\prk_{\cG\cup(\cH_i)_{[d-k]}} (F) \le Ar^B$ for $i = 1,2$, where 
    \[
    \cH_1 = \cG_1\cup\ldots\cup \cG_{2^k},\ \cH_2 = \cG_{2^k+1}\cup\ldots\cup \cG_{2^{k+1}}. 
    \]
    Now apply lemma \ref{lem:gluing-two} to get that $\prk_{\cG\cup(\cH_1\cup\cH_2)_{[d-k-1]}} (F) \le A_1(A_kr^{B_k})^{B_1}$, which completes the inductive step.
\end{proof}

We can now finish our proof.

\begin{proof}[Proof of lemma \ref{lem:gluing-deriv}]
    By lemma \ref{lem:str-prime}, the variety $Z(\cG_I)^{2^dm}$ is an irreducible complete intersection, so by proposition \ref{dense-ml} its $K$-points are dense. By lemma \ref{lem:reg-der}, the tower $\cG\cup\cG(\ul{y}^{(1)})\cup\ldots\cup \cG(\ul{y}^{(2^d)})$ is absolutely $(A,B,r)$-strong for a dense open subset of $\ul{y} \in Z(\cG_I)^{2^dm}$. Choose $\ul{y}^{(1)},\ldots,\ul{y}^{(2^d)}$ to be $K$-points in this open subset which also satisfy $\prk_{\cG\cup \cG(\ul{y}^{(i)})} (F) \le r$ for all $i\in [2^d]$. By corollary \ref{cor:gluing-many} we get $\prk_{\cG}(F) \le Ar^B$ as claimed.   
\end{proof}

\bibliographystyle{plain}
\bibliography{refs}

@ARTICLE{BDS,
       author = {{Bik}, Arthur and {Draisma}, Jan and {Snowden}, Andrew},
        title = "{Two improvements in Brauer's theorem on forms}",
      journal = {arXiv e-prints},
         year = 2024,
        month = jan,
          eid = {arXiv:2401.02067},
        pages = {arXiv:2401.02067},
          doi = {10.48550/arXiv.2401.02067},
archivePrefix = {arXiv},
       eprint = {2401.02067},
 primaryClass = {math.NT},
       adsurl = {https://ui.adsabs.harvard.edu/abs/2024arXiv240102067B}
}

@article {Birch,
    AUTHOR = {Birch, B. J.},
     TITLE = {Forms in many variables},
   JOURNAL = {Proc. Roy. Soc. London Ser. A},
  FJOURNAL = {Proceedings of the Royal Society. London. Series A.
              Mathematical, Physical and Engineering Sciences},
    VOLUME = {265},
      YEAR = {1961/62},
     PAGES = {245--263},
      ISSN = {0962-8444,2053-9169},
   MRCLASS = {10.15},
  MRNUMBER = {150129},
MRREVIEWER = {John\ V.\ Armitage},
       DOI = {10.1098/rspa.1962.0007},
       URL = {https://doi-org.proxy.lib.umich.edu/10.1098/rspa.1962.0007},
}

@ARTICLE{L-small,
       author = {{Lampert}, Amichai},
        title = "{Small ideals in polynomial rings and applications}",
      journal = {arXiv e-prints},
         year = 2023,
        month = sep,
          eid = {arXiv:2309.16847},
        pages = {arXiv:2309.16847},
          doi = {10.48550/arXiv.2309.16847},
archivePrefix = {arXiv},
       eprint = {2309.16847},
 primaryClass = {math.AC},
       adsurl = {https://ui.adsabs.harvard.edu/abs/2023arXiv230916847L}
}

@article {KLP,
    AUTHOR = {Kazhdan, David and Lampert, Amichai and Polishchuk, Alexander},
     TITLE = {Schmidt rank and singularities},
      NOTE = {Reprint of Ukra\"in. Mat. Zh. {\bf 75} (2023), no. 9,
              1248--1266},
   JOURNAL = {Ukrainian Math. J.},
  FJOURNAL = {Ukrainian Mathematical Journal},
    VOLUME = {75},
      YEAR = {2024},
    NUMBER = {9},
     PAGES = {1420--1442},
      ISSN = {0041-5995,1573-9376},
   MRCLASS = {14N07},
  MRNUMBER = {4720725},
}

@ARTICLE{LS-Birch,
       author = {{Lampert}, Amichai and {Snowden}, Andrew},
        title = "{Two improvements in Birch's theorem on forms}",
      journal = {arXiv e-prints},
         year = 2024,
        month = jun,
          eid = {arXiv:2406.18498},
        pages = {arXiv:2406.18498},
          doi = {10.48550/arXiv.2406.18498},
archivePrefix = {arXiv},
       eprint = {2406.18498},
 primaryClass = {math.NT},
}

@ARTICLE{Lam-density,
       author = {{Lampert}, Amichai},
        title = "{Density of solutions for systems of forms}",
      journal = {arXiv e-prints},
         year = 2025,
        month = jul,
          eid = {arXiv:2507.11514},
        pages = {arXiv:2507.11514},
          doi = {10.48550/arXiv.2507.11514},
archivePrefix = {arXiv},
       eprint = {2507.11514},
 primaryClass = {math.NT},
       adsurl = {https://ui.adsabs.harvard.edu/abs/2025arXiv250711514L}
}

@article {Peck,
    AUTHOR = {Peck, L. G.},
     TITLE = {Diophantine equations in algebraic number fields},
   JOURNAL = {Amer. J. Math.},
  FJOURNAL = {American Journal of Mathematics},
    VOLUME = {71},
      YEAR = {1949},
     PAGES = {387--402},
      ISSN = {0002-9327,1080-6377},
}

@article {Schmidt,
    AUTHOR = {Schmidt, Wolfgang M.},
     TITLE = {The density of integer points on homogeneous varieties},
   JOURNAL = {Acta Math.},
  FJOURNAL = {Acta Mathematica},
    VOLUME = {154},
      YEAR = {1985},
    NUMBER = {3-4},
     PAGES = {243--296}
}

@article {Wooley,
    AUTHOR = {Wooley, Trevor D.},
     TITLE = {On the local solubility of {D}iophantine systems},
   JOURNAL = {Compositio Math.},
  FJOURNAL = {Compositio Mathematica},
    VOLUME = {111},
      YEAR = {1998},
    NUMBER = {2},
     PAGES = {149--165},
      ISSN = {0010-437X,1570-5846},
   MRCLASS = {11D72 (11D88 11E76)},
  MRNUMBER = {1606240},
MRREVIEWER = {R.\ C.\ Baker},
       DOI = {10.1023/A:1000298711968},
       URL = {https://doi.org/10.1023/A:1000298711968},
}

@article {Birch-odd,
    AUTHOR = {Birch, B. J.},
     TITLE = {Homogeneous forms of odd degree in a large number of
              variables},
   JOURNAL = {Mathematika},
  FJOURNAL = {Mathematika. A Journal of Pure and Applied Mathematics},
    VOLUME = {4},
      YEAR = {1957},
     PAGES = {102--105},
      ISSN = {0025-5793},
   MRCLASS = {10.00},
  MRNUMBER = {97359},
MRREVIEWER = {G.\ Whaples},
       DOI = {10.1112/S0025579300001145},
       URL = {https://doi.org/10.1112/S0025579300001145},
}

@article {Brauer,
    AUTHOR = {Brauer, Richard},
     TITLE = {A note on systems of homogeneous algebraic equations},
   JOURNAL = {Bull. Amer. Math. Soc.},
  FJOURNAL = {Bulletin of the American Mathematical Society},
    VOLUME = {51},
      YEAR = {1945},
     PAGES = {749--755},
      ISSN = {0002-9904},
   MRCLASS = {09.0X},
  MRNUMBER = {13127},
MRREVIEWER = {G.\ Whaples},
       DOI = {10.1090/S0002-9904-1945-08440-7},
       URL = {https://doi.org/10.1090/S0002-9904-1945-08440-7},
}

@article {LS,
    AUTHOR = {Leep, D. B. and Schmidt, W. M.},
     TITLE = {Systems of homogeneous equations},
   JOURNAL = {Invent. Math.},
  FJOURNAL = {Inventiones Mathematicae},
    VOLUME = {71},
      YEAR = {1983},
    NUMBER = {3},
     PAGES = {539--549},
      ISSN = {0020-9910,1432-1297},
   MRCLASS = {11D72 (11E99)},
  MRNUMBER = {695905},
MRREVIEWER = {D.\ J.\ Lewis},
       DOI = {10.1007/BF02095992},
       URL = {https://doi.org/10.1007/BF02095992},
}

@article {Dav-32,
    AUTHOR = {Davenport, H.},
     TITLE = {Cubic forms in thirty-two variables},
   JOURNAL = {Philos. Trans. Roy. Soc. London Ser. A},
  FJOURNAL = {Philosophical Transactions of the Royal Society of London.
              Series A. Mathematical and Physical Sciences},
    VOLUME = {251},
      YEAR = {1959},
     PAGES = {193--232},
      ISSN = {0080-4614},
   MRCLASS = {10.00},
  MRNUMBER = {105394},
MRREVIEWER = {L.\ Carlitz},
       DOI = {10.1098/rsta.1959.0002},
       URL = {https://doi.org/10.1098/rsta.1959.0002},
}

@article {Dav-29,
    AUTHOR = {Davenport, H.},
     TITLE = {Cubic forms in {$29$} variables},
   JOURNAL = {Proc. Roy. Soc. London Ser. A},
  FJOURNAL = {Proceedings of the Royal Society. London. Series A.
              Mathematical, Physical and Engineering Sciences},
    VOLUME = {266},
      YEAR = {1962},
     PAGES = {287--298},
      ISSN = {0962-8444,2053-9169},
   MRCLASS = {10.17 (10.25)},
  MRNUMBER = {136580},
MRREVIEWER = {B.\ J.\ Birch},
       DOI = {10.1098/rspa.1962.0062},
       URL = {https://doi.org/10.1098/rspa.1962.0062},
}

@article {Dav-16,
    AUTHOR = {Davenport, H.},
     TITLE = {Cubic forms in sixteen variables},
   JOURNAL = {Proc. Roy. Soc. London Ser. A},
  FJOURNAL = {Proceedings of the Royal Society. London. Series A.
              Mathematical, Physical and Engineering Sciences},
    VOLUME = {272},
      YEAR = {1963},
     PAGES = {285--303},
      ISSN = {0962-8444,2053-9169},
   MRCLASS = {10.17},
  MRNUMBER = {155800},
MRREVIEWER = {B.\ J.\ Birch},
       DOI = {10.1098/rspa.1963.0054},
       URL = {https://doi.org/10.1098/rspa.1963.0054},
}

@article {HB-14,
    AUTHOR = {Heath-Brown, D. R.},
     TITLE = {Cubic forms in 14 variables},
   JOURNAL = {Invent. Math.},
  FJOURNAL = {Inventiones Mathematicae},
    VOLUME = {170},
      YEAR = {2007},
    NUMBER = {1},
     PAGES = {199--230},
      ISSN = {0020-9910,1432-1297},
   MRCLASS = {11E76},
  MRNUMBER = {2336082},
MRREVIEWER = {Timothy\ D.\ Browning},
       DOI = {10.1007/s00222-007-0062-1},
       URL = {https://doi.org/10.1007/s00222-007-0062-1},
}

@book {Eisenbud,
    AUTHOR = {Eisenbud, David},
     TITLE = {Commutative algebra},
    SERIES = {Graduate Texts in Mathematics},
    VOLUME = {150},
      NOTE = {With a view toward algebraic geometry},
 PUBLISHER = {Springer-Verlag, New York},
      YEAR = {1995},
     PAGES = {xvi+785},
      ISBN = {0-387-94268-8; 0-387-94269-6},
   MRCLASS = {13-01 (14A05)},
  MRNUMBER = {1322960},
MRREVIEWER = {Matthew\ Miller},
       DOI = {10.1007/978-1-4612-5350-1},
       URL = {https://doi.org/10.1007/978-1-4612-5350-1},
}

@article {Skinner-HLS,
    AUTHOR = {Skinner, C. M.},
     TITLE = {Forms over number fields and weak approximation},
   JOURNAL = {Compositio Math.},
  FJOURNAL = {Compositio Mathematica},
    VOLUME = {106},
      YEAR = {1997},
    NUMBER = {1},
     PAGES = {11--29},
      ISSN = {0010-437X,1570-5846},
   MRCLASS = {14G05 (11G35 11P55)},
  MRNUMBER = {1446148},
MRREVIEWER = {Dan\ Abramovich},
       DOI = {10.1023/A:1000129818730},
       URL = {https://doi.org/10.1023/A:1000129818730},
}

@article {LZ-rel,
    AUTHOR = {Lampert, Amichai and Ziegler, Tamar},
     TITLE = {Relative rank and regularization},
   JOURNAL = {Forum Math. Sigma},
  FJOURNAL = {Forum of Mathematics. Sigma},
    VOLUME = {12},
      YEAR = {2024},
     PAGES = {Paper No. e29, 26},
      ISSN = {2050-5094},
   MRCLASS = {11B30 (11T06 13F20 14N07 15A03)},
  MRNUMBER = {4715068},
MRREVIEWER = {Jonathan\ Chapman},
       DOI = {10.1017/fms.2024.15},
       URL = {https://doi.org/10.1017/fms.2024.15},
}

@article {Sch-cubics,
    AUTHOR = {Schmidt, Wolfgang M.},
     TITLE = {On cubic polynomials. {IV}. {S}ystems of rational equations},
   JOURNAL = {Monatsh. Math.},
  FJOURNAL = {Monatshefte f\"ur Mathematik},
    VOLUME = {93},
      YEAR = {1982},
    NUMBER = {4},
     PAGES = {329--348},
      ISSN = {0026-9255,1436-5081},
   MRCLASS = {10B10},
  MRNUMBER = {666834},
MRREVIEWER = {H.\ G.\ Meijer},
       DOI = {10.1007/BF01295233},
       URL = {https://doi.org/10.1007/BF01295233},
}

@article {Diet-cubics,
    AUTHOR = {Dietmann, Rainer},
     TITLE = {On the {$h$}-invariant of cubic forms, and systems of cubic
              forms},
   JOURNAL = {Q. J. Math.},
  FJOURNAL = {The Quarterly Journal of Mathematics},
    VOLUME = {68},
      YEAR = {2017},
    NUMBER = {2},
     PAGES = {485--501},
      ISSN = {0033-5606,1464-3847},
   MRCLASS = {11D72 (11E76)},
  MRNUMBER = {3667211},
MRREVIEWER = {Julia\ Brandes},
       DOI = {10.1093/qmath/haw048},
       URL = {https://doi.org/10.1093/qmath/haw048},
}

@incollection {Wool-quintic,
    AUTHOR = {Wooley, Trevor D.},
     TITLE = {Forms in many variables},
 BOOKTITLE = {Analytic number theory ({K}yoto, 1996)},
    SERIES = {London Math. Soc. Lecture Note Ser.},
    VOLUME = {247},
     PAGES = {361--376},
 PUBLISHER = {Cambridge Univ. Press, Cambridge},
      YEAR = {1997},
      ISBN = {0-521-62512-2},
   MRCLASS = {11E76 (11D72)},
  MRNUMBER = {1695003},
MRREVIEWER = {Michael\ A.\ Bennett},
       DOI = {10.1017/CBO9780511666179.025},
       URL = {https://doi.org/10.1017/CBO9780511666179.025},
}

@article {FM,
    AUTHOR = {Frei, Christopher and Madritsch, Manfred},
     TITLE = {Forms of differing degrees over number fields},
   JOURNAL = {Mathematika},
  FJOURNAL = {Mathematika. A Journal of Pure and Applied Mathematics},
    VOLUME = {63},
      YEAR = {2017},
    NUMBER = {1},
     PAGES = {92--123},
      ISSN = {0025-5793,2041-7942},
   MRCLASS = {11G35 (11P55 14G05)},
  MRNUMBER = {3610007},
MRREVIEWER = {Damaris\ Schindler},
       DOI = {10.1112/S0025579316000206},
       URL = {https://doi-org.proxy.lib.umich.edu/10.1112/S0025579316000206},
}

@article {Mil,
    AUTHOR = {Mili\'cevi\'c, Luka},
     TITLE = {Polynomial bound for partition rank in terms of analytic rank},
   JOURNAL = {Geom. Funct. Anal.},
  FJOURNAL = {Geometric and Functional Analysis},
    VOLUME = {29},
      YEAR = {2019},
    NUMBER = {5},
     PAGES = {1503--1530},
}

@article {BHB,
    AUTHOR = {Browning, Tim and Heath-Brown, Roger},
     TITLE = {Forms in many variables and differing degrees},
   JOURNAL = {J. Eur. Math. Soc. (JEMS)},
  FJOURNAL = {Journal of the European Mathematical Society (JEMS)},
    VOLUME = {19},
      YEAR = {2017},
    NUMBER = {2},
     PAGES = {357--394},
      ISSN = {1435-9855,1435-9863},
   MRCLASS = {11P55 (11G35 14G05)},
  MRNUMBER = {3605019},
MRREVIEWER = {Joseph\ H.\ Silverman},
       DOI = {10.4171/JEMS/668},
       URL = {https://doi-org.proxy.lib.umich.edu/10.4171/JEMS/668},
}

@article {Wool-explicit,
    AUTHOR = {Wooley, Trevor D.},
     TITLE = {An explicit version of {B}irch's theorem},
   JOURNAL = {Acta Arith.},
  FJOURNAL = {Acta Arithmetica},
    VOLUME = {85},
      YEAR = {1998},
    NUMBER = {1},
     PAGES = {79--96},
      ISSN = {0065-1036,1730-6264},
   MRCLASS = {11D72 (11E76)},
  MRNUMBER = {1623369},
MRREVIEWER = {G.\ Greaves},
       DOI = {10.4064/aa-85-1-79-96},
       URL = {https://doi.org/10.4064/aa-85-1-79-96},
}

@article {Birch-p-adic,
    AUTHOR = {Birch, B. J.},
     TITLE = {Diagonal equations over {$p$}-adic fields},
   JOURNAL = {Acta Arith.},
  FJOURNAL = {Polska Akademia Nauk. Instytut Matematyczny. Acta Arithmetica},
    VOLUME = {9},
      YEAR = {1964},
     PAGES = {291--300},
      ISSN = {0065-1036},
   MRCLASS = {10.14},
  MRNUMBER = {167456},
MRREVIEWER = {D.\ J.\ Lewis},
       DOI = {10.4064/aa-9-3-291-300},
       URL = {https://doi.org/10.4064/aa-9-3-291-300},
}

\end{document}